\documentclass[reqno]{amsart}

\usepackage{amsmath}

\usepackage{amsthm}
\usepackage{amssymb}
\usepackage{graphics}
\usepackage{latexsym}
\usepackage{comment}
\usepackage{extarrows}
\usepackage{wrapfig}
\usepackage{comment}
\usepackage {colortbl,array,xcolor}
\usepackage{hhline}
\usepackage{tikz}
\usetikzlibrary{intersections, calc, math, patterns}

\numberwithin{equation}{section}
\newtheorem{thm}{Theorem}[section]
\newtheorem{prop}[thm]{Proposition}
\newtheorem{lem}[thm]{Lemma}

\newtheorem{defn}[thm]{Definition}
\theoremstyle{remark}

\newcommand{\R}{{\mathbb R}}
\newcommand{\Z}{{\mathbb Z}}

\newcommand{\N}{{\mathbb N}}

\newcommand{\LR}[1]{{\langle {#1} \rangle }}

\newcommand{\F}{\mathcal{F}}

\newcommand{\ha}{\widehat}

\newcommand{\supp}{\operatorname{supp}}

\title[The ZK equation in high d: Small initial data of critical regularity]{The
  Zakharov-Kuznetsov equation in high dimensions:
  Small initial data of critical regularity}

\author[S.~Herr]{Sebastian Herr}
\address[Sebastian Herr]{Universit\"{a}t Bielefeld\\
Fakult\"{a}t f\"{u}r Mathematik\\
Postfach 10 01 31\\
33501 Bielefeld\\
Germany}
\email{herr@math.uni-bielefeld.de}

\author[S.  Kinoshita]{Shinya Kinoshita}
\address[Shinya Kinoshita]{Universit\"{a}t Bielefeld\\
Fakult\"{a}t f\"{u}r Mathematik\\
Postfach 10 01 31\\
33501 Bielefeld\\
Germany}
\email{kinoshita@math.uni-bielefeld.de}

\subjclass[2010]{35Q53, 35A01}
\keywords{global well-posedness, scattering, Zakharov-Kuznetsov equation}
\thanks{
Financial support by the
  German Research Foundation (DFG) through the CRC 1283 ``Taming uncertainty and profiting from
  randomness and low regularity in analysis, stochastics and their
  applications'' is acknowledged.
}
\begin{document}
\begin{abstract}
The Zakharov-Kuznetsov equation
in spatial dimension $d\geq 5$ is considered. The Cauchy problem is
shown to be globally well-posed for small initial data in critical
spaces and it is proved that solutions scatter to free solutions as
$t \to \pm \infty$. The proof is based on i) novel endpoint non-isotropic
Strichartz estimates which are derived from the $(d-1)$-dimensional
Schr\"odinger equation, ii) transversal bilinear restriction
estimates,
and iii) an interpolation argument in critical
function spaces. Under an additional radiality assumption, a similar result is obtained in dimension $d=4$.
\end{abstract}
\maketitle

\section{Introduction}\label{sec:intro}
This paper is concerned with the Zakharov--Kuznetsov equation
\begin{equation}\label{eq:ZK}
  \begin{split}
  \partial_t u+\partial_{x_1}\Delta u={}&\partial_{x_1}u^2 \phantom{u_0} \text{ in
  }\R\times \R^d\\
  u(0,\cdot)={}&u_0  \phantom{\partial_{x_1}u^2 } \text{ on
  } \R^d
\end{split}
\end{equation}
where $d\geq 2$, $u=u(t,x)$, $(t,x) =(t,x_1,\ldots, x_d) \in \R
\times \R^d$,  $u$ is real-valued, and $\Delta$ denotes the Laplacian with respect to $x$.

The Zakharov--Kuznetsov equation was introduced in \cite{ZK74} as a model for  propagation of ion-sound waves in magnetic fields.
The Zakharov--Kuznetsov equation can be seen as a multidimensional extension of the well-known Korteweg--de Vries (KdV) equation.
In contrast to the KdV equation, the Zakharov--Kuznetsov equation is not completely integrable, but possesses two invariants,
\[
M(u) := \int_{\R^d} u^2 dx, \quad
E(u) :=  \int_{\R^d} \frac{1}{2} |\nabla u|^2 + \frac{1}{3} u^3 dx.
\]

In the following, $H^s(\R^d)$ denotes the standard $L^2$-based inhomogeneous Sobolev space
and $B^s_{2,1}(\R^d)$ is the Besov refinement, and the dotted versions their homogeneous counterparts, see below for definitions.
The scale-invariant regularity threshold for \eqref{eq:ZK} is $s_c=\frac{d-4}{2}$.

Before we state our main results, let us briefly summarize the progress which has been made regarding the well-posedness problem associated to \eqref{eq:ZK}. In the two-dimensional case, Faminski\u{\i} \cite{Fa95} established global well-posedness in the energy space $H^1(\R^2)$. Later, Linares and Pastor \cite{LP09} proved local well-posedness in $H^s(\R^2)$ for $s>3/4$, before Gr\"{u}nrock and Herr \cite{GH14} and Molinet and Pilod
\cite{MP15} showed local well-posedness for $s > 1/2$. Recently, the second author \cite{Ki2019} proved local well-posedness for $s>-1/4$.
In dimension $d=3$, Linares and Saut \cite{LS09} obtained local well-posedness in $H^s(\R^3)$ for $s >
9/8$. Ribaud and Vento \cite{RV12-ZK} proved local well-posedness for $s > 1$ and in $B_2^{1,1}(\R^3)$. The global well-poseness in $H^s(\R^3)$ for $s>1$ was obtained by Molinet and Pilod in \cite{MP15}.
Recently, in dimensions $d
\geq 3$, local well-posedness in $H^s(\R^d)$ in the full subcritical range $s>s_c$ was proved in \cite{HK20arxiv}, which implies global well-posedness in $H^1(\R^d)$ if $3
\leq d \leq 5$ and in $L^2(\R^3)$. We refer the reader to these papers for a more thorough account on the Zakharov--Kuznetsov equation, and more references.

In the present paper, we address the problem of global well-posedness and scattering for small initial data in critical spaces.
By well-posedness we mean existence of a (mild) solution, uniqueness of solutions (in some subspace) and (locally Lipschitz) continuous dependence of solutions on the initial data.
We say that a global solution $u \in C(\R,H^s(\R^d))$ of \eqref{eq:ZK} scatters as $t \to \pm \infty$, if there exist $u_\pm \in H^s(\R^d)$ such that
\[\| u(t)-e^{tS}u_\pm\|_{H^s(\R^d)}\to 0 \quad (t \to \pm \infty).\]
Here, $e^{tS}$ denotes the unitary group generated by the skew-adjoint linear
operator $S = -\partial_{x_1} \Delta$, so that $e^{tS}u_\pm$ solves the linear homogeneous equation.

Our first main result covers small data in dimension $d=5$.
\begin{thm}\label{thm:d5}
For $d =5$, the Cauchy problem \eqref{eq:ZK} is globally well-posed for small initial data in $B^{s_c}_{2,1}(\R^5)$, and solutions scatter as  $t\to\pm \infty$. The same result holds in $\dot{B}^{s_c}_{2,1}(\R^d)$.
\end{thm}

In dimensions $d\geq 6$, we can extend this to Sobolev regularity.
\begin{thm}\label{thm:d6}
For $d \geq 6$, the Cauchy problem \eqref{eq:ZK} is globally well-posed for small initial data in $H^{s_c}(\R^d)$, and solutions scatter as $t\to\pm \infty$. The same result holds in $\dot{H}^{s_c}(\R^d)$.
\end{thm}
Note that in $d=6$ this result includes the energy space $\dot{H}^1(\R^6)$.

If we restrict to  initial data which is radial in the last $(d-1)$ variables (see below for definitions), we obtain small data global well-posedness and scattering in the critical Sobolev spaces for any dimension $d\geq 4$.

\begin{thm}\label{thm:d4rad}
For $d \geq 4$, the Cauchy problem \eqref{eq:ZK} is globally well-posed for small data in $H^{s_c}_{\mathrm{rad}}(\R^d)$, and solutions scatter as $t\to \pm \infty$.  The same result holds for radial data in $\dot{H}^{s_c}_{\mathrm{rad}}(\R^d)$.
\end{thm}
As the proof shows, the radiality assumption can be weakened to an angular regularity assumption, but we do not pursue this.
One of the most interesting special cases here is $d=4$, when $s_c=0$, hence the result covers the radial $L^2$ space.

The main idea of this paper is to combine a new set of non-isotropic Strichartz estimates with the bilinear transversal estimate and an interpolation argument in critical function spaces.

The paper is structured as follows: In Subsection \ref{subsec:not} we introduce notation. In Section \ref{sec:str} we derive Strichartz type estimates which are based on the well-known Strichartz estimates for the $(d-1)$-dimensional Schr\"odinger equation and allow us to treat the case $d=5$. In Section \ref{sec:d6} we combine this with the bilinear transversal estimate and an interpolation argument, which leads to a proof of Theorem \ref{thm:d6}. Finally, in Section \ref{sec:rad} we discuss an variation of these ideas under the additional radiality assumption and a proof of Theorem \ref{thm:d4rad}.

\subsection{Notation}\label{subsec:not}

We write $ x'=(x_2,\ldots,x_d)$, $D_{x_j}=-i \partial_j$,
$D=(-\Delta)^\frac12$, $|\nabla_{x'}|^s=\F_{x'}^{-1} |\xi'|^s \F_{x'}$, and $\langle \nabla_{x'}\rangle^s=\F_{x'}^{-1} \langle \xi'\rangle^s \F_{x'}$. Here and in the sequel we denote the Fourier transform of $u$
in time, space, and the first spatial variable, by $\F_t u$, $\F_{x}
u$, $\F_{x_1} u$, and  $\F_{x'}$, respectively.
$\F_{t, x} u = \ha{u}$ denotes the Fourier transform of $u$ in time
and space. Choose a non-negative bump function $\psi \in C_c^\infty(\R)$ supported in the interval $(1/2,2)$ with the property that $\sum_{N \in 2^\Z} \psi(r/N)=1$ for $r>0$, and set $\psi_N=\psi(\cdot/N)$.
For $N$, $\lambda \in 2^{\Z}$, we define (spatial) frequency projections $P_N$, $Q_{\lambda}$ as the Fourier multipliers with symbols $\psi_{N}(|\xi| )$, $\psi_{\lambda}(|\xi_1| )$, and $ \psi_{M}(|\xi'| )$, respectively,
where $(\tau,\xi)=(\tau,\xi_1,\xi')=(\tau,\xi_1,\ldots,\xi_d) \in \R \times \R^d$ are temporal and spatial frequencies.
In addition, we define \[P_{\leq 1}=\sum_{1\geq N \in 2^\Z}P_N, \quad Q_{\leq 1}:= \sum_{1\geq  \lambda \in 2^\Z} Q_\lambda, \text{ and } R_{\leq 1}:= \sum_{1\geq  M \in 2^\Z} R_M.
\]
As usual, the Sobolev space $H^s(\R^d)$ is defined as the completion of the Schwartz functions with respect to the norm
\[
\|f\|_{H^s(\R^d)}=\Big(\int_{\R^d} \LR{\xi}^{2s} |\widehat{f}(\xi)|^2 d\xi\Big)^{\frac12},
\]
and the (smaller) Besov space $B^s_{2,1}(\R^d)$ as the completion of the Schwartz functions $\mathcal{S}(\R^{d})$ with respect to the norm
\[
\|f\|_{B^s_{2,1}(\R^d)}=\|P_{\leq 1}f\|_{L^2} +\sum_{N\in 2^{\N} }N^s\|P_Nf\|_{L^2}.
\]
Similarly, for $s\geq 0$, the homogeneous Sobolev space $\dot{H}^s(\R^d)$ is defined as the completion of the Schwartz functions with respect to the norm
\[
\|f\|_{\dot{H}^s(\R^d)}=\Big(\int_{\R^d} |\xi|^{2s} |\widehat{f}(\xi)|^2 d\xi\Big)^{\frac12},
\]
and the homogeneous Besov space $\dot{B}^s_{2,1}(\R^d)$ as the completion of the Schwartz functions with respect to the norm
\[
\|f\|_{\dot{B}^s_{2,1}(\R^d)}=\sum_{N\in 2^{\Z} }N^s\|P_Nf\|_{L^2}.
\]

The radial subspaces $H_{\mathrm{rad}}^s(\R^d)$ and $\dot{H}^s_{\mathrm{rad}}(\R^d)$ are defined by the requirement that $f(x_1,x') = f (x_1, y')$ if $ |x'|=|y'|$, i.e.\ for fixed $x_1$, the functions are radial in $x'$.

Finally,  the Duhamel operator is denoted by
\begin{equation*}
\mathcal{I}(F)(t) := \int_0^t e^{(t-t')S} F(t') d t'.
\end{equation*}

\section{Strichartz estimates and the proof of Theorem \ref{thm:d5}}\label{sec:str}

For $d \geq 2$, we say $(q, r)$ is \emph{$(d-1)$-admissible} if
\[2 \leq q,r \leq \infty, \; 2/q = (d-1)(1/2-1/r), \; (d, q, r) \not= (3,2,\infty).\]
\begin{thm}\label{th:Strichartz}
Let $d \geq 2$ and  $(q_1, r_1)$, $(q_2,r_2)$ be $(d-1)$-admissible. Then, we have
\begin{align}
& \| D_{x_1}^{\frac1q} e^{t S} f \|_{L_t^{q_1} L_{x'}^{r_1} {{L}_{x_1}^2}} \lesssim \| f \|_{L_x^2}.
\label{eq:homogeneousStrichartz}\\
& \|D_{x_1}^{\frac{1}{q_1}+\frac{1}{q_2}}\mathcal{I} F\|_{L_t^{q_1} L_{x'}^{r_1} {{L}_{x_1}^2}} \lesssim \|F\|_{L_t^{q_2'} L_{x'}^{r_2'} {{L}_{x_1}^2}},\label{eq:retardedStrichartz}
\end{align}
where $1/q_2'=1-1/q_2$ and $1/r_2'=1-1/r_2$.
\end{thm}
\begin{proof}
Let $\Delta_{x'} = \sum_{j=2}^d \partial_{x_j}^2$.
For fixed $\xi_1 \in \R$, define $V_{\xi_1}(t) f(x') := (e^{-i t \xi_1 \Delta_{x'}} f)(x')$.
Since $V_{\xi_1}(t/\xi_1) = e^{it \Delta_{x'}}$, for $f \in \mathcal{S}(\R^{d-1})$ and $F \in\mathcal{S}(\R \times \R^{d-1})$, the Strichartz estimates of Schr\"{o}dinger equations in $\R^{d-1}$ imply
\begin{align}
\label{eq:Strichartzd-1}
\| |\xi_1|^{\frac{1}{q_1}} V_{\xi_1}(t) f \|_{L_t^{q_1} L_{x'}^{r_1}}
  \lesssim{}& \| f \|_{L_{x'}^2},
  \\
\Bigl\| \int_0^t |\xi_1|^{\frac{1}{q_1}+\frac{1}{q_2}} V_{\xi_1}(t-t') F(t') d t' \Bigr\|_{L_t^{q_1} L_{x'}^{r_1}}
  \lesssim{}& \| F\|_{L_t^{q_2'} L_{x'}^{r_2'}},
              \label{eq:Strichartzd-12}
\end{align}
see \cite[Theorem 1.2]{Keel-Tao} for details. We deduce from the Plancherel's Theorem, Minkowski's inequality and \eqref{eq:Strichartzd-1} that
\[
 \| D_{x_1}^{\frac1q} e^{t S} f \|_{L_t^q L_{x'}^r L_{x_1}^2}
 = \Bigl( \int_{\R} \| |\xi_1|^{\frac1q} V_{\xi_1} \F_{x_1} f  \|_{L_t^q L_{x'}^r}^{2} d \xi_1 \Bigr)^{\frac12} \lesssim \| f \|_{L_x^2},\]
which is \eqref{eq:homogeneousStrichartz}. Similarly, by \eqref{eq:Strichartzd-12},
\begin{align*}
 &\|D_{x_1}^{\frac{1}{q_1}+\frac{1}{q_2}} \mathcal{I}
   (F)\|_{L_t^{q_1} L_{x'}^{r_1} L_{x_1}^2} \\
  ={}& \Bigl( \int_{\R}
\Bigl\| \int_0^t |\xi_1|^{\frac{1}{q_1}+\frac{1}{q_2}} V_{\xi_1}(t-t') \F_{x_1}(F)(t') d t' \Bigr\|_{L_t^{q_1} L_{x'}^{r_1}}^2  d \xi_1 \Bigr)^{\frac12}\\
\lesssim{}& \|F\|_{L_t^{q_2'} L_{x'}^{r_2'} L_{x_1}^2},
\end{align*}
which is \eqref{eq:retardedStrichartz}.
\end{proof}

Now, we can complete the proof of Theorem \ref{thm:d5}. Recall that $d=5$ implies $s_c = 1/2$.

\begin{defn}
We define
\begin{align*}
   \| u \|_{Z^{\frac12}} :={}& \| P_{\leq 1}u \|_{L_t^{\infty} L_x^2} + \|P_{\leq 1} D_{x_1}^{\frac12} u\|_{L_t^2 L_{x'}^4 L_{x_1}^2} \\
  {}&+ \sum_{N \in 2^{\N}}N^{\frac12} \bigl( \| P_{N}u \|_{L_t^{\infty} L_x^2} + \|P_{N} D_{x_1}^{\frac12} u\|_{L_t^2 L_{x'}^4 L_{x_1}^2} \bigr), \\
 \| u \|_{\dot{Z}^{\frac12}} :={}& \sum_{N \in 2^{\Z}}N^{\frac12} \bigl( \| P_{N}u \|_{L_t^{\infty} L_x^2} + \|P_{N} D_{x_1}^{\frac12} u\|_{L_t^2 L_{x'}^4 L_{x_1}^2} \bigr),
\end{align*}
and the corresponding Banach spaces.
\end{defn}
By the standard argument involving the contraction mapping principle, it suffices to prove the following:
\begin{prop}
Let $d=5$. Then, we have
\[
\| \mathcal{I} (\partial_{x_1}(u_1u_2)) \|_{Z^{\frac12}} \lesssim \|u_1 \|_{Z^{\frac12}} \| u_2 \|_{Z^{\frac12}}, \quad
\| \mathcal{I} (\partial_{x_1}(u_1u_2)) \|_{\dot{Z}^{\frac12}} \lesssim \|u_1 \|_{\dot{Z}^{\frac12}} \| u_2 \|_{\dot{Z}^{\frac12}}.
\]
\end{prop}
\begin{proof}
Let $N_{\max} = \max (N_1,N_2,N_3)$ and $N_{\min} = \min (N_1,N_2, N_3)$.
For any $N \in 2^\Z$, Theorem \ref{th:Strichartz} gives
\[
\| P_{N}\mathcal{I} (\partial_{x_1}(u_1u_2)) \||_{L_t^{\infty} L_x^2} +
\| P_{N}D_{x_1}^{\frac12} \mathcal{I} (\partial_{x_1}(u_1u_2)) \|_{L_t^2 L_{x'}^4 L_{x_1}^2} \lesssim
\|P_N D_{x_1}^{\frac12} (u_1 \, u_2) \|_{L_t^2 L_{x'}^{\frac43} L_{x_1}^2}.
\]
Further, we obtain
\begin{align*}
& \|P_{N_3} D_{x_1}^{\frac12} (u_{N_1} \, u_{N_2}) \|_{L_t^2 L_{x'}^{\frac43} L_{x_1}^2} \\
& \lesssim N_{\min}^{\frac12}\|D_{x_1}^{\frac12} u_{N_1} \|_{L_t^2 L_{x'}^4 L_{x_1}^2} \|u_{N_2} \|_{L_t^{\infty} L_x^2} + N_{\min}^{\frac12} \|u_{N_1} \|_{L_t^{\infty} L_x^2} \|D_{x_1}^{\frac12} u_{N_2} \|_{L_t^2 L_{x'}^4 L_{x_1}^2}
\end{align*}
from the Kato-Ponce inequality and the Bernstein inequality. This can be summed up both in the homogeneous and in the inhomogeneous version.
\end{proof}
This argument also implies the scattering claim, since it implies that the Duhamel integral converges to a free solutions as $t\to \pm \infty$. We omit the details of this standard argument.

\section{Transversal estimates and the proof of Theorem \ref{thm:d6}}\label{sec:d6}

\begin{lem}\label{lem:trans}
Let $d\geq 2$ and $f_{N_1,\lambda_1} = Q_{\lambda_1} P_{N_1} f$, $g_{N_2,\lambda_2} = Q_{\lambda_2} P_{N_2} g$. For all $\lambda_j,N_j\in 2^\Z$ such that
\[
|\nabla \varphi(\xi)-\nabla \varphi(\eta)|\gtrsim \max\{\lambda_1,\lambda_2\} N_{\max},
\]
for all $\xi \in \supp \widehat{f}_{N_1,\lambda_1}$, $\eta \in \supp \widehat{g}_{N_2,\lambda_2}$, it holds that
\begin{equation}\label{eq:transversest}
\|P_{N_3} (e^{tS}f_{N_1,\lambda_1} \, e^{tS}g_{N_2,\lambda_2}) \|_{L_t^2 L_x^2} \lesssim \Big(\frac{N_{\min}^{d-1}}{  \max\{\lambda_1,\lambda_2\} N_{\max}}\Big)^{\frac12} \|f_{N_1,\lambda_1}\|_{L^2}
\|g_{N_2,\lambda_2} \|_{L^2}.
\end{equation}
\end{lem}
This is an instance of the well-known  bilinear transversal estimate, e.g. a special case of \cite[Lemma 2.6]{CH18}, where a proof can be found.

Next, we recall the definitions of $U^p$ and $V^p$ spaces, which have been introduced in \cite{KT05} the dispersive PDE context. We refer the reader to \cite{CH18} and the references therein for further details.
For $1\leq p <\infty$, we call a function $a: \R \to L^2(\R^d)$ a $p-$atom, if there exists a finite partition $\mathcal{J}=\{(-\infty,t_1),[t_2,t_3), \ldots, [t_K,\infty)\}$ of the real line such that
\[a(t)=\sum_{J \in \mathcal{J}}\mathbf{1}(t)f_J, \quad \sum_{J \in \mathcal{J}}\|f_J\|_{L^2}^p\leq 1.\]
Now, $U^p$ is defined as the space of all $u: \R \to L^2(\R^d)$, such that there exists an atomic decomposition $u=\sum_{j=1 }^\infty c_j a_j$, where $(c_j)\in \ell^1(\N)$ and the $a_j$'s are $p-$atoms. Then, $\|u\|_{U^p}=\inf \sum_{j=1}^\infty |c_j|$ is a norm (the infimum is taken with respect to all possible atomic decompositions), so that $U^p$ is a Banach space. Further, let $V^p$ denote the space of all right-continuous functions $v: \R \to L^2(\R^d)$, such that
\[
\|v\|_{V^p}=\|v\|_{L^\infty_t L^2_x}+\sup \Big(\sum_{j \in \Z}\|v(t_j)-v(t_{j-1})\|_{L^2_x}^p\Big)^{\frac1p}<\infty,
\]
where the supremum is taken over all increasing sequences $(t_j)$.
Now, we define the atomic space $U_{S}^p = e^{\cdot \, S} U^p$ with norm $\|u \|_{U_S^p} = \|e^{- \, \cdot \, S} u\|_{U^p}$, and
 $V_{S}^p = e^{\cdot \, S} V^p$ with norm $\|u \|_{V_S^p} = \|e^{- \, \cdot \, S} u\|_{V^p}$.

There is the embedding $V_{S}^p\subset U_{S}^q$ if $p<q$, see \cite[Lemma 6.4]{KT05}. Due to the atomic structure of $U_{S}^q$ and the Strichartz estimate \eqref{eq:homogeneousStrichartz}, we have
\begin{equation}\label{eq:linearest1}
\| D_{x_1}^{\frac{1}{q}} u\|_{L_t^{q} L_{x'}^{r} L_{x_1}^2} \lesssim \| u \|_{U^{q}_S}
\end{equation}
for $(d-1)-$admissible pairs, and  $\| u \|_{U^{q}_S}$ may be replaced by  $\| u \|_{V^{2}_S}$ for non-endpoint pairs, i.e. when $q>2$.

Let $\lambda_{\max}  := \max (\lambda_1,\lambda_2,\lambda_3)$ and
$\lambda_{\min} := \min (\lambda_1,\lambda_2,\lambda_3)$.  We use the shorthand notation $u_N := P_{N} u$, $u_{N, \lambda} := Q_{\lambda} P_{N} u $, etc.
\begin{prop}\label{prop:inter}
Let $d\geq 6$ and the pair $(q,r)$ be $(d-1)$-admissible with  $2<q<\frac{2(d-3)}{d-5}$, and let $\varepsilon >0$. Suppose
\[
|\nabla \varphi(\xi)-\nabla \varphi(\eta)|\gtrsim \max\{\lambda_1,\lambda_2\} N_{\max},
\]
for all $\xi \in \supp_{\xi} \widehat{u}_{N_1,\lambda_1}$, $\eta \in \supp_{\xi} \widehat{v}_{N_2,\lambda_2}$. Then, for all $\lambda_j,N_j \in 2^\Z$,
\begin{equation}\label{eq:bilinearest2}
\begin{split}
& \|  P_{N_3} Q_{\lambda_3} (  u_{N_1,\lambda_1} v_{N_2,\lambda_2}) \|_{L_t^{q'} L_{x'}^{r'} L_{x_1}^2} \\
& \lesssim \lambda_{\max}^{-\frac12+  \frac{2 \varepsilon}{d-3}}
\lambda_{\min}^{-\frac{1}{2}+\frac{1}{d-3}+\frac{1}{q}-\frac{2 \varepsilon}{d-3}} N_{\max}^{-\frac12 + \frac{1}{d-3}} N_{\min}^{\frac{d-3}{2} - \frac{2}{d-3}}
\|u_{N_1,\lambda_1} \|_{V^2_S}
\| v_{N_2,\lambda_2}\|_{V^2_S}.
\end{split}
\end{equation}
\end{prop}

\begin{proof}
By symmetry, we may assume that $\lambda_1 \sim \lambda_{\max}$. For a sufficiently small $\varepsilon>0$, we define the $(d-1)$-admissible pairs $(q_1,r_1)$ and $(q_2,r_2)$ by
\[
\Bigl( \frac{1}{q_1}, \ \frac{1}{r_1} \Bigr) = \Bigl(\frac12 -\varepsilon, \frac12-\frac{1- 2\varepsilon}{d-1} \Bigr), \
\Bigl(\frac{1}{q_2}, \ \frac{1}{r_2}\Bigr) = \Bigl(\frac{d-3}{4}- \frac{d-3}{2 q} + \varepsilon, \
\frac{d-3}{q(d-1)} + \frac{1- 2 \varepsilon}{d-1} \Bigr).
\]
In addition, letting
\[
\frac{1}{\alpha} = \frac{1}{q_1} + \frac{1}{q_2} , \qquad
\frac{1}{\beta} = \frac{1}{r_1} + \frac{1}{r_2},
\]
by using \eqref{eq:linearest1}, we have
\begin{align}
\|  P_{N_3}Q_{\lambda_3}(  u_{N_1,\lambda_1} v_{N_2,\lambda_2}) \|_{L_t^{\alpha} L_{x'}^{\beta} L_{x_1}^2} & \lesssim \lambda_{\min}^{\frac12}
\| u_{N_1,\lambda_1} \|_{L_t^{q_1} L_{x'}^{r_1} L_{x_1}^2}
\| v_{N_2,\lambda_2}\|_{L_t^{q_2} L_{x'}^{r_2} L_{x_1}^2}\notag \\
& \lesssim \lambda_{\min}^{\frac12-\frac{1}{q_2}} \lambda_{\max}^{-\frac{1}{q_1}} \|u_{N_1,\lambda_1} \|_{U^{q_1}_S}
\| v_{N_2,\lambda_2}\|_{U_S^{q_2}}.\label{eq:bilinearest1}
\end{align}
Lemma \ref{lem:trans} immediately extends from free solutions to $2-$atoms. Therefore, the atomic structure of $U^2$ implies
\begin{equation}\label{eq:transversest2}
\|P_{N_3}Q_{\lambda_3}(  u_{N_1,\lambda_1} v_{N_2,\lambda_2}) \|_{L_t^2 L_x^2} \lesssim \Big(\frac{N_{\min}^{d-1}}{ \lambda_{\max} N_{\max}}\Big)^{\frac12}  \|u_{N_1,\lambda_1} \|_{U^{2}_S}
\| v_{N_2,\lambda_2}\|_{U_S^{2}}.
\end{equation}
For $\theta = \frac{2}{d-3}$, it is observed that
\[
\frac{\theta}{\alpha} + \frac{1- \theta}{2} = 1 - \frac{1}{q} \Bigl( =: \frac{1}{q'} \Bigr), \quad
\frac{\theta}{\beta} + \frac{1- \theta}{2} = 1 - \frac{1}{r} \Bigl(=: \frac{1}{r'} \Bigr).
\]
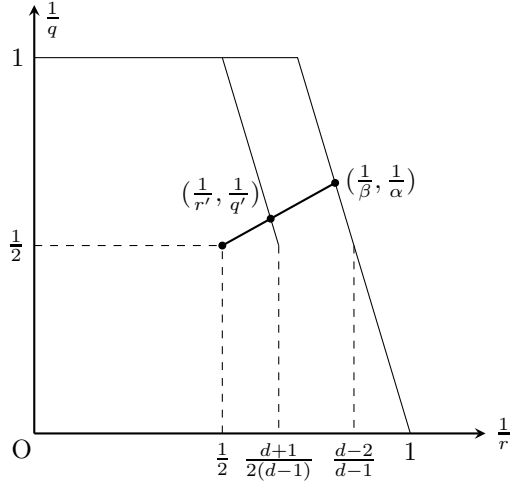
\begin{figure}\label{fig:inter}
  \caption{Choice of parameters in the interpolation argument}
\begin{tikzpicture}[thick]
  \useasboundingbox (0,-0.8) rectangle (7,6);
   \draw [thick, -stealth](0.5,0)--(6.5,0) node [anchor=west, font=\normalsize]{$\frac{1}{r}$};
   \draw [thick, -stealth](0.5,0)--(0.5,5.7) node [anchor=west, font=\small]{};
\node [anchor=west] at(0.5,5.5){$\frac{1}{q}$};
\node [anchor=east] at(0.5,5){$1$};
\node [anchor=east] at(0.6,-0.2){O};
\fill (3,2.5) circle (1.5pt);
\draw [thin, dashed] (0.5,2.5)--(3,2.5);
\draw [thin,dashed] (3,0)--(3,2.5);
\node [anchor=east] at(0.5,2.5){$\frac12$};
\node [anchor=north] at(3,0){$\frac12$};
\draw [thin](0.5,5)--(4,5)--(5.5,0);
\draw [thin](3,5)--(3.75,2.5);
\draw [thin, dashed](4.75,2.5)--(4.75,0);
\node [anchor=north, font=\normalsize] at(4.75,0){$\frac{d-2}{d-1}$};
\node [anchor=north, font=\normalsize] at(5.5,0){$1$};
\draw [thin, dashed](3.75,2.5)--(3.75,0);
\node [anchor=north, font=\normalsize] at(3.75,0){$\frac{d+1}{2(d-1)}$};
\draw [thick](3,2.5)--(4.5,10/3) node [anchor=west, font=\normalsize]{$(\frac{1}{\beta}, \frac{1}{\alpha})$};
\fill (4.5,10/3) circle (1.5pt);
\fill (3+9/14,2.5+5/14) circle (1.5pt);
\node [anchor=east, font=\normalsize] at(3+9/14,2.5+5/14+0.3){$(\frac{1}{r'},\frac{1}{q'})$};
\end{tikzpicture}
\end{figure}
Now, we interpolate \eqref{eq:bilinearest1} and \eqref{eq:transversest2} to obtain \eqref{eq:bilinearest2}. More precisely, we follow the argument in \cite[p.~1203]{CH18}: For brevity, we set $u:=u_{N_1,\lambda_1} $, $v:=v_{N_2,\lambda_2} $. Then, \cite[Lemma 6.4]{KT05} implies that there exist decompositions $u=\sum_{k=1}^\infty u_{k}$, such that $\widehat{u_{k}}\subset \supp \widehat{u}$, and for any $q \geq 2$ we have
$
\|u_{k}\|_{U^{q}_S}\lesssim 2^{k(\frac{2}{q}-1)}\|u\|_{V^2_S},
$
and the analogous decomposition for $v$.
Then, by convexity,  we obtain
\begin{align*}
  \|  P_{N_3} Q_{\lambda_3} (  u v) \|_{L_t^{q'} L_{x'}^{r'} L_{x_1}^2} \lesssim {}& \sum_{k,k'\in \N}\|  P_{N_3} Q_{\lambda_3} (  u_k v_{k'}) \|_{L_t^{q'} L_{x'}^{r'} L_{x_1}^2} \\
  \lesssim {}& \sum_{k,k'\in \N}\|  P_{N_3} Q_{\lambda_3} (  u_k v_{k'}) \|_{L_t^{\alpha} L_{x'}^{\beta} L_{x_1}^2}^\theta\|  P_{N_3} Q_{\lambda_3} (  u_k v_{k'}) \|_{L^2_{t,x}}^{1-\theta}.
\end{align*}
Estimates \eqref{eq:bilinearest1} and \eqref{eq:transversest2} further imply
\begin{align*}
\|  P_{N_3} Q_{\lambda_3} (  u v) \|_{L_t^{q'} L_{x'}^{r'} L_{x_1}^2} \lesssim {}& \Big(\sum_{k,k'\in \N}  2^{k\theta (\frac{2}{q_1}-1)}2^{k'\theta(\frac{2}{q_2}-1)}\Big) \frac{\lambda_{\min}^{\frac{\theta}{2}-\frac{\theta}{q_2}} N_{\min}^{\frac{d-1}{2}(1-\theta)}}{ \lambda_{\max}^{\frac{\theta}{q_1}+\frac{1-\theta}2} N_{\max}^{\frac{1-\theta}2}}  \|u\|_{V^2_S}\|v\|_{V^2_S}  .
\end{align*}
Since $q_1,q_2>2$, the sums converge, and the proof of \eqref{eq:bilinearest2} is complete.
\end{proof}

Define
\[
\mathcal{R}_{\lambda_{\max},N_{\max}} = \{ (\xi_1, \xi') \in \R \times \R^{d-1} \, | \, |\xi_1|\ll \lambda_{\max}, \ |\xi'| \ll N_{\max} \}.
\]
\begin{lem}\label{lem2.5}
Assume that there exist $\gamma_1$, $\gamma_2$, $\gamma_3 \in \R^{d}$ such that $\gamma_1 + \gamma_2 -\gamma_3 \in \mathcal{R}_{4 \lambda_{\max}, 4N_{\max}}$ and
\begin{equation}\label{assumption:inclusions}
\supp_{\xi} \widehat{u}_{N_i, \lambda_i} \subset \mathcal{R}_{\lambda_{\max},N_{\max}} + \gamma_i :=
\{\xi \in \R^d \, | \, \xi - \gamma_i \in \mathcal{R}_{\lambda_{\max},N_{\max}} \} .
\end{equation}
Then we have either
\begin{equation}\label{est:transversality2}
|\nabla \varphi(\xi) - \nabla \varphi (\eta)| \gtrsim \lambda_{\max}N_{\max},
\end{equation}
for all $\xi \in \supp_{\xi} \widehat{u}_{N_1, \lambda_1}$, $\eta \in \supp_{\xi} \widehat{u}_{N_2, \lambda_2}$, or
\begin{equation}
|\nabla \varphi(\eta)-\nabla \varphi(\zeta)| \gtrsim \lambda_{\max}N_{\max},
\end{equation}
for all $\eta \in \supp_{\xi} \widehat{u}_{N_2, \lambda_2}$, $\zeta \in \supp_{\xi} \widehat{u}_{N_3, \lambda_3}$.
\end{lem}
\begin{proof}
Firstly, we consider the case $\max(|\xi'|, |\eta'|, |\zeta'|) \ll N_{\max}$.
We deduce from $\partial_1 \varphi(\xi)= 3 \xi_1^2 + |\xi'|^2$ that
\begin{align*}
& |\partial_1 \varphi(\xi) - \partial_1 \varphi (\eta)| + | \partial_1 \varphi (\eta) - \partial_1 \varphi(\xi + \eta)|\\
 & \geq 3|\xi_1^2 - \eta_1^2| + 3|\xi_1(\xi_1+2\eta_1)| - |\xi'|^2-|\eta'|^2-|\xi'+\eta'|^2 \gtrsim N_{\max}^2,
\end{align*}
which implies the claim since $|\nabla^2 \varphi(\xi)| \lesssim |\xi|$.

Next we assume $\max(|\xi'|, |\eta'|, |\zeta'|) \sim N_{\max}$. For all $\xi \in \supp_{\xi} \widehat{u}_{N_1, \lambda_1}$, $\eta \in \supp_{\xi} \widehat{u}_{N_2, \lambda_2}$, $\xi + \eta \in \supp_{\xi} \widehat{u}_{N_3, \lambda_3}$, we will show
\begin{equation}\label{est:transversality}
\sum_{j=2}^d \bigl(|\partial_j \varphi(\xi) - \partial_j \varphi (\eta)| + |\partial_j \varphi(\eta)-\partial_j \varphi(\xi + \eta)| \bigr) \gtrsim \lambda_{\max}N_{\max}.
\end{equation}
We may assume $|\xi'| \sim N_{\max}$, $\lambda_1 \sim \lambda_{\max}$.
For $2 \leq j \leq d$, it is observed that $\partial_j \varphi (\xi) = 2 \xi_1 \xi_j$. Then, for \eqref{est:transversality}, it suffices to show
\begin{equation}\label{est:transversality2-lem2.5}
|\xi_1 \xi' - \eta_1 \eta'| + |\eta_1 \eta' - (\xi_1+\eta_1)(\xi' + \eta')| \gtrsim \lambda_1 N_1.
\end{equation}
Since $|\xi'| \sim N_{\max}$, $\lambda_1 \sim \lambda_{\max}$, if either $\lambda_{\min} \ll \lambda_{\max}$ or $\min(|\eta'|, |\xi'+\eta'|) \ll N_{\max}$ holds, we easily verify \eqref{est:transversality2-lem2.5}. Then we assume $\lambda_1 \sim \lambda_2 \sim \lambda_3$ and $|\xi'| \sim |\eta'| \sim |\xi'+\eta'|$. We observe
\begin{align*}
|\eta_1 \eta' - (\xi_1+\eta_1)(\xi' + \eta')| & = \bigl|\eta_1 \eta' - (\xi_1 + \eta_1)
\bigl( \xi' - \frac{\eta_1}{\xi_1} \eta' + \frac{\eta_1}{\xi_1} \eta' + \eta' \bigr)\bigr|\\
& \geq \bigl| \eta_1 \eta' - (\xi_1 + \eta_1)
\bigl( 1 + \frac{\eta_1}{\xi_1}\bigr) \eta' \bigr| - \bigl| \bigl(1+\frac{\eta_1}{\xi_1}\bigr)
(\xi_1 \xi' - \eta_1 \eta')\bigr|\\
& = \bigl| 1+\frac{\eta_1}{\xi_1}+ \frac{\xi_1}{\eta_1} \, \bigr| |\eta_1 \eta'| - \bigl| \bigl(1+\frac{\eta_1}{\xi_1}\bigr)
(\xi_1 \xi' - \eta_1 \eta')\bigr|.
\end{align*}
Since $|\alpha + \alpha^{-1}| \geq 2$ for any $\alpha \in \R$,
this completes the proof of \eqref{est:transversality2-lem2.5}.

From \eqref{est:transversality}, without loss of generality, we can assume that there exist $\xi_0 \in \supp_{\xi} \widehat{u}_{N_1, \lambda_1}$, $\eta_0 \in \supp_{\xi} \widehat{u}_{N_2, \lambda_2}$ such that
\begin{equation}\label{est:transversality3-lem2.5}
\sum_{j=2}^d |\partial_j \varphi(\xi_0) - \partial_j \varphi (\eta_0)| \gtrsim \lambda_{\max}N_{\max}.
\end{equation}
For $2 \leq j, k \leq d$ and all $\xi \in \supp_{\xi} \widehat{u}_{N_1, \lambda_1}$, $\eta \in \supp_{\xi} \widehat{u}_{N_2, \lambda_2}$, since $|\partial_1 \partial_j \varphi (\xi)| + |\partial_1 \partial_j \varphi (\eta)| \lesssim N_{\max}$ and
$|\partial_k \partial_j \varphi (\xi)| +|\partial_k \partial_j \varphi (\eta)| \lesssim \lambda_{\max}$, we get
\[
| \partial_j \varphi(\xi) - \partial_j \varphi(\xi_0)| + | \partial_j \varphi(\eta) - \partial_j \varphi(\eta_0)| \ll
\lambda_{\max} N_{\max},
\]
for all $\xi \in \supp_{\xi} \widehat{u}_{N_1, \lambda_1}$, $\eta \in \supp_{\xi} \widehat{u}_{N_2, \lambda_2}$.
This estimate and \eqref{est:transversality3-lem2.5} yield the claim \eqref{est:transversality2}.
\end{proof}

Now we define the solution spaces as
$Y^s:=C(\R; H^s(\R^d) \cap \LR{\nabla_x}^{-s} V^2_S$ and $ \dot{Y}^s:=C(\R; \dot{H}^s(\R^d) \cap {|\nabla_x|}^{-s} V^2_S$, with norms
\begin{align*}
& \| u \|_{Y^s} := \Bigl( \sum_{N \in 2^{\Z}}  \langle N\rangle^{2 s} \| P_N  u \|_{V^2_S}^2 \Bigr)^{1/2} , \\
& \| u \|_{\dot{Y}^s} := \Bigl( \sum_{N \in 2^{\Z}}  N^{2 s} \| P_N u \|_{V^2_S}^2 \Bigr)^{1/2},
\end{align*}
respectively.

\begin{prop}\label{prop:d6}
Let $d \geq 6$. Then we have
\[
\| \mathcal{I} (\partial_{x_1}(u_1u_2)) \|_{Y^{s_c}} \lesssim \|u_1 \|_{Y^{s_c}} \| u_2 \|_{Y^{s_c}}, \quad
\| \mathcal{I} (\partial_{x_1}(u_1u_2)) \|_{\dot{Y}^{s_c}} \lesssim \|u_1 \|_{\dot{Y}^{s_c}} \| u_2 \|_{\dot{Y}^{s_c}}.
\]
\end{prop}
\begin{proof}
We show first that there exists $\varepsilon >0$ such that for any $N_1,N_2,N_3\in 2^\Z$ we have
\begin{equation}\label{eq:prop2.2-1}
\Bigl| \iint P_{N_1} u_1 P_{N_2} u_2 \partial_{x_1} P_{N_3} u_3 dx dt \Bigr|
\lesssim N_{\min}^{s_c + \varepsilon} N_{\max}^{- \varepsilon}  \prod_{i=1}^3 \| P_{N_i}u_{i} \|_{V^2_S}.
\end{equation}
As before, we use the shorthand notation $u_{N_j} := P_{N_j} u_j$, $u_{N_j, \lambda_j} := Q_{\lambda_j} P_{N_j} u_j $, etc. Obviously, \eqref{eq:prop2.2-1} is implied by
\begin{equation}\label{eq:prop2.2-2}
\sum_{\lambda_1, \lambda_2, \lambda_{3} \in 2^{\Z}}\lambda_3 \Bigl| \iint u_{N_1,\lambda_1} u_{N_2,\lambda_2}  u_{N_3,\lambda_3} dx dt \Bigr|
\lesssim N_{\min}^{s_c + \varepsilon} N_{\max}^{- \varepsilon}  \prod_{i=1}^3 \| u_{N_i} \|_{V^2_S}.
\end{equation}

Now we show \eqref{eq:prop2.2-2}.
After harmless decompositions, we may assume that there exist $\gamma_1$, $\gamma_2$, $\gamma_3 \in \R^{d}$ such that
$\gamma_1 + \gamma_2 -\gamma_3 \in \mathcal{R}_{4 \lambda_{\max}, 4N_{\max}}$ and  \eqref{assumption:inclusions}.
Lemma \ref{lem2.5} provides either $|\nabla \varphi(\xi)-\nabla \varphi(\eta)| \gtrsim \lambda_{\max}N_{\max}$ for all $\xi \in \supp_{\xi} u_{N_1,\lambda_1}$, $\eta \in \supp_{\xi} u_{N_2,\lambda_2}$ or $|\nabla \varphi(\eta)-\nabla \varphi(\zeta)| \gtrsim \lambda_{\max}N_{\max}$ for all $\eta \in \supp_{\xi} u_{N_2,\lambda_2}$ and $\zeta \in \supp_{\xi} u_{N_3,\lambda_3}$.
For the former case, it follows from the H\"{o}lder's inequality, the Strichartz estimate \eqref{eq:linearest1}, and the bilinear estimate \eqref{eq:bilinearest2} that
\begin{align*}
& \sum_{\lambda_1, \lambda_2, \lambda_{3} \in 2^{\Z}}\lambda_3
\Bigl| \iint u_{N_1,\lambda_1} u_{N_2,\lambda_2}  u_{N_3,\lambda_3} dx dt \Bigr|\\
 & \leq \sum_{\lambda_i \leq N_i (i=1,2,3)}\lambda_3
\|  P_{N_3} Q_{\lambda_3} (  u_{N_1,\lambda_1} u_{N_2,\lambda_2}) \|_{L_t^{q'} L_{x'}^{r'} L_{x_1}^2} \|u_{N_3,\lambda_3}\|_{L_t^{q} L_{x'}^{r} L_{x_1}^2}\\
& \lesssim \sum_{\lambda_i \leq N_i (i=1,2,3)} \lambda_{\max}^{\frac{d-1}{d-3}\varepsilon}
\lambda_{\min}^{\frac{1}{d-3}-\frac{d-1}{d-3} \varepsilon} N_{\max}^{-\frac12 + \frac{1}{d-3}} N_{\min}^{\frac{d-3}{2} - \frac{2}{d-3}}
\|u_{N_1,\lambda_1} \|_{V^2_S}
\| u_{N_2,\lambda_2}\|_{V_S^{2}} \|u_{N_3,\lambda_3}\|_{V^2_S}\\
& \leq N_{\min}^{s_c + \frac12 -\frac{1}{d-3}- \frac{d-1}{d-3} \varepsilon}N_{\max}^{- \frac12 + \frac{1}{d-3}+\frac{d-1}{d-3} \varepsilon}
\|u_{N_1} \|_{V^2_S}
\| u_{N_2}\|_{V_S^{2}} \|u_{N_3}\|_{V^2_S}.
\end{align*}
Here, the pair $(q,r)$ should satisfy the hypothesis of Proposition \ref{prop:inter}, and we have used $\lambda_{\max} \leq N_{\max}$ and
$\lambda_{\min}\leq N_{\min}$.
In the similar way, the latter case is treated as follows:
\begin{align*}
& \sum_{\lambda_1, \lambda_2, \lambda_{3} \in 2^{\Z}}\lambda_3
\Bigl| \iint u_{N_1,\lambda_1} u_{N_2,\lambda_2}  u_{N_3,\lambda_3} dx dt \Bigr|\\
 & \leq \sum_{\lambda_i \leq N_i (i=1,2,3)} \lambda_3
\|  P_{N_1} Q_{\lambda_1} (  u_{N_2,\lambda_2} u_{N_3,\lambda_3}) \|_{L_t^{q'} L_{x'}^{r'} L_{x_1}^2} \|u_{N_1,\lambda_1}\|_{L_t^{q} L_{x'}^{r} L_{x_1}^2}\\
& \leq N_{\min}^{s_c + \frac12 -\frac{1}{d-3}- \frac{d-1}{d-3} \varepsilon}N_{\max}^{- \frac12 + \frac{1}{d-3}+\frac{d-1}{d-3} \varepsilon}
\|u_{N_1} \|_{V^2_S}
\| u_{N_2}\|_{V_S^{2}} \|u_{N_3}\|_{V^2_S}.
\end{align*}
Finally, we explain why \eqref{eq:prop2.2-1} implies Proposition \ref{prop:d6}. By duality, see e.g.\ \cite[Lemma 7.3]{CH18}, we obtain
\begin{align*}
\|P_{N_3} \mathcal{I}\big(\partial_{x_1}( P_{N_1} u_1 P_{N_2} u_2)\big) \|_{V^2_S}
\lesssim N_{\min}^{s_c}\Big(\frac{N_{\min}}{N_{\max}}\Big)^{ \varepsilon}  \| P_{N_1}u_{1} \|_{V^2_S}\| P_{N_2}u_{2} \|_{V^2_S}.
\end{align*}
This can be easily summed up.
\end{proof}
Again, the proof of Theorem \ref{thm:d6} is a straight-forward application of the contraction mapping principle. The scattering claim follows from the well-known fact that functions in $V^2$ have limits at $\pm \infty$.

\section{Radial Strichartz estimates and the proof of Theorem \ref{thm:d4rad}}\label{sec:rad}

We first prove a variant of the Strichartz estimates in \ref{th:Strichartz} for functions which, for fixed $x_1$, are radial in $x'$.
\begin{thm}\label{thm:RadStrichartz}
Let $d \geq 3$ and  $2 \leq q$,$r \leq \infty$ satisfy
\[\frac{2}{q} \leq (2 d -3) \Bigl( \frac12 - \frac1r \Bigr), \quad (d,q,r)\not=(3,2,\infty), \quad (q,r)\not=\Big(2,\frac{2(2d-3)}{2d-5}\Big),
\]
and let $\sigma=-\frac{d-1}2+\frac{d-1}r +\frac2q$.
Then, for all functions $f \in L^2_{\textnormal{rad}}(\R^d)$, we have
\begin{equation}
\| D_{x_1}^{\frac1q} |\nabla_{x'}|^\sigma e^{t S} f \|_{L_t^{q} L_{x'}^{r} {{L}_{x_1}^2}} \lesssim \| f \|_{L_x^2}.
\label{eq:homogeneousRadStrichartz}
\end{equation}
\end{thm}
The proof follows the exact same lines as the proof of Theorem \ref{th:Strichartz}, but with the Strichartz estimates for the $(d-1)$-dimensional Schr\"odinger equation from \cite{Keel-Tao} replaced by the radial version obtained in \cite[Theorem 1.1]{CL13}.
\begin{lem}\label{lem:Radtrans}
Let $d\geq 2$ and $f_{N_1,\lambda_1,M_1} = R_{M_1} Q_{\lambda_1} P_{N_1} f$, $g_{N_2,\lambda_2,M_2} = R_{M_2} Q_{\lambda_2} P_{N_2} g$. (i) Suppose that there exists $\ell \in \{2,\ldots,d\}$ such that
\[
|\partial_{\ell} \varphi(\xi)-\partial_{\ell} \varphi(\eta)|\gtrsim N_{\max}^2,
\]
for all $\xi \in \supp \widehat{f}_{N_1,\lambda_1,M_1}$, $\eta \in \supp \widehat{g}_{N_2,\lambda_2,M_2}$. Then it holds that
\begin{equation}\label{eq:Radtransversest1}
\begin{split}
& \|P_{N_3}(e^{tS}f_{\lambda_1,N_1,M_1} \, e^{tS}g_{\lambda_2,N_2,M_2}) \|_{L_t^2 L_x^2} \\
& \lesssim \Big(\frac{\min \{\lambda_1,\lambda_2 \} \min \{M_{1},M_2\}^{d-2}}{ N_{\max}^2}\Big)^{\frac12} \|f_{M_1,\lambda_1,N_1}\|_{L^2}
\|g_{M_2,\lambda_2,N_2} \|_{L^2}.
\end{split}
\end{equation}
(ii) Suppose that 
\[
|\partial_{1} \varphi(\xi)-\partial_{1} \varphi(\eta)|\gtrsim N_{\max}^2,
\]
for all $\xi \in \supp \widehat{f}_{N_1,\lambda_1,M_1}$, $\eta \in \supp \widehat{g}_{N_2,\lambda_2,M_2}$. Then it holds that
\begin{equation}\label{eq:Radtransversest2}
\begin{split}
& \|P_{N_3}(e^{tS}f_{\lambda_1,N_1,M_1} \, e^{tS}g_{\lambda_2,N_2,M_2}) \|_{L_t^2 L_x^2} \\
& \lesssim \Big(\frac{ \min \{M_{1},M_2\}^{d-1}}{ N_{\max}^2}\Big)^{\frac12} \|f_{M_1,\lambda_1,N_1}\|_{L^2}
\|g_{M_2,\lambda_2,N_2} \|_{L^2}.
\end{split}
\end{equation}
\end{lem}
As above, the proof of this lemma follows from \cite[Lemma 2.6]{CH18}. As above, it immediately extends to $U^2_S$-functions.

Let $Y^{s}_{\textnormal{rad}}$ and $\dot{Y}^{s}_{\textnormal{rad}}$ be the subspaces of  $Y^{s}$ and $\dot{Y}^{s}$ of functions which, for fixed $x_1$, are radial in $x'$, with the same norms. Then, the key for the proof of Theorem \ref{thm:d4rad} is the following
\begin{prop}\label{prop:bil4drad}
Let $d \geq 4$. Then we have
\[
\| \mathcal{I} (\partial_{x_1}(u_1u_2)) \|_{Y^{s_c}_{\textnormal{rad}}} \lesssim \|u_1 \|_{Y^{s_c}_{\textnormal{rad}}} \| u_2 \|_{Y^{s_c}_{\textnormal{rad}}}, \quad
\| \mathcal{I} (\partial_{x_1}(u_1u_2)) \|_{\dot{Y}^{s_c}_{\textnormal{rad}}} \lesssim \|u_1 \|_{\dot{Y}^{s_c}_{\textnormal{rad}}} \| u_2 \|_{\dot{Y}^{s_c}_{\textnormal{rad}}}.
\]
\end{prop}
\begin{proof}
For $i=1,2,3$, we use $u_{i} := R_{M_i} Q_{\lambda_i} P_{N_i} u $. As in the proof of Proposition \ref{prop:d6}, it suffices to show
\begin{equation}\label{eq:prop4.2-01}
\sum_{\lambda_i, M_i} \lambda_{\max} \Bigl| \iint u_1 u_{2}
u_{3} dx dt \Bigr|
\lesssim N_{\min}^{s_c + \varepsilon} N_{\max}^{- \varepsilon}  \prod_{i=1}^3 \| u_{N_i} \|_{V^2_S}.
\end{equation}
Here and in the sequel, all functions are implicitely assumed to satisfy the radiality hypothesis.
Let
\begin{align*}
& \Bigl(\frac{1}{q_1}, \frac{1}{r_1} \Bigr) = \Bigl( \frac12 - \varepsilon, \frac{2d-5}{2(2d-3)} \Bigr),\\
& \Bigl(\frac{1}{q_2}, \frac{1}{r_2} \Bigr) = \Bigl( \frac{(d-1)(2 d-3)}{2 (d-1+ 2 (2d-3) \varepsilon)} \varepsilon, \frac{d-1}{2 (d-1+ 2 (2d-3) \varepsilon)} \Bigr),\\
& \Bigl(\frac{1}{q_3}, \frac{1}{r_3} \Bigr) = \Big(2 \varepsilon, \frac{2}{2d-3}\Bigr).
\end{align*}
Then we have
\begin{align}
& \| R_{M} Q_{\lambda} P_{N} u \|_{L_t^{q_1} L_{x'}^{r_1} L_{x_1}^2} \lesssim \lambda^{- \frac{1}{q_1}} M^{-\frac{d-2}{2 d-3} + 2 \varepsilon} \|R_{M} Q_{\lambda} P_{N} u\|_{U^{q_1}_S},\label{est:RadStrichartz-1}\\
& \| R_{M} Q_{\lambda} P_{N} u\|_{L_t^{q_2} L_{x'}^{r_2} L_{x_1}^2} \lesssim \lambda^{- \frac{1}{q_2}} \|R_{M} Q_{\lambda} P_{N} u\|_{U^{q_2}_S},\label{est:RadStrichartz-2}\\
& \| R_{M} Q_{\lambda} P_{N} u\|_{L_t^{q_3} L_{x'}^{r_3} L_{x_1}^2} \lesssim \lambda^{- \frac{1}{q_3}} M^{\frac{(d-1)(2 d-7)}{2 (2 d-3)} - 4 \varepsilon}\|R_{M} Q_{\lambda} P_{N} u\|_{U^{q_3}_S}.\label{est:RadStrichartz-3}
\end{align}
By symmetry of \eqref{eq:prop4.2-01} we may assume $N_3 \lesssim N_1 \sim N_2$, $\lambda_2 \lesssim \lambda_1$, and then it is enough to consider the following three cases:
\\
(1) $M_1 \sim N_1$, $M_2 \sim N_2$. (2) $M_1 \sim N_1$, $M_2 \ll N_1$, (3) $M_1 \ll N_1$, $M_2 \ll N_1$.

(1) First, we assume $M_1 \sim N_1$, $M_2 \sim N_2$. By using \eqref{est:RadStrichartz-1} and \eqref{est:RadStrichartz-3} we obtain
\begin{align*}
&\Bigl| \iint u_1 u_{2}
u_{3} dx dt \Bigr|\\
 &  \lesssim \lambda_{\min}^{\frac12} \|u_1 \|_{L_t^{q_1} L_{x'}^{r_1} L_{x_1}^2}
\|u_{2} \|_{L_t^{q_1} L_{x'}^{r_1} L_{x_1}^2} \|u_{3} \|_{L_t^{q_3} L_{x'}^{r_3} L_{x_1}^2} \\
& \lesssim \lambda_{\min}^{\frac12} \lambda_1^{-\frac{1}{q_1}} \lambda_2^{-\frac{1}{q_1}} \lambda_3^{-2 \varepsilon} M_3^{\frac{(d-1)(2 d-7)}{2 (2 d-3)} - 4 \varepsilon} N_1^{-\frac{2(d-2)}{2d-3}+4 \varepsilon} \prod_{i=1,2,3} \|u_i \|_{V^2_S}\\
& \lesssim \lambda_{\min}^{\varepsilon} \lambda_{\max}^{-1}M_3^{s_c+\varepsilon}N_1^{-2 \varepsilon} \prod_{i=1,2,3} \|u_i \|_{V^2_S},
\end{align*}
which completes \eqref{eq:prop4.2-01}.

(2) In the case $M_1 \sim N_1$, $M_2 \ll N_1$, it is observed that $\lambda_2 \sim M_3 \sim N_1$. Then, without loss of generality, we may assume $\lambda_3 \lesssim \lambda_1 \sim N_1$.
In the case $\lambda_3 \ll \lambda_1$, for all $\xi \in \supp_{\xi} \widehat{u}_{1}$, $\eta \in \supp_{\xi} \widehat{u}_{2}$ such that $\xi+\eta \in \supp_{\xi} \widehat{u}_3$, we observe
\begin{align*}
& |\tau_1- \varphi(\xi)|+|\tau_2 - \varphi(\eta)|+ |\tau_1+\tau_2 - \varphi(\xi+\eta)|\\
 \gtrsim{} &
|\varphi(\xi+\eta) - \varphi(\xi) - \varphi(\eta)| \\
\gtrsim{} & \bigl| \xi_1|\xi'|^2 \bigr| - |\xi_1+\eta_1| \, \bigl| |\xi+ \eta|^2 + \xi_1^2 - \xi_1 \eta_1 + \eta_1^2 \bigr| - \bigl|\eta_1 | \eta'|^2 \bigr|  \gtrsim N_1^3.
\end{align*}
Thus we can assume that at least one of $u_1$, $u_2$, $u_3$ satisfies $\supp \widehat{u}_i \subset \{ (\tau,\xi)\, |\, |\tau-\varphi(\xi)| \gtrsim N_1^3\}$.
We easily see that this condition verifies the claim by utilizing Theorem \ref{th:Strichartz} and \eqref{eq:linearest1}.
For example, if  $\supp \widehat{u}_1 \subset \{ (\tau,\xi)\, |\, |\tau-\varphi(\xi)| \gtrsim N_1^3\}$, using Bernstein's inequality and Theorem \ref{thm:RadStrichartz} we obtain
\begin{align*}
\Bigl| \iint u_1 u_{2}  u_{3} dx dt \Bigr| & \lesssim \|u_1 \|_{L_t^2 L_x^2} \|u_2 \|_{L_t^4 L_{x'}^{2(d-1)} L_{x_1}^2}  \| u_3\|_{L_t^4 L_{x'}^{\frac{2(d-1)}{d-2}} L_{x_1}^{\infty}}\\
& \lesssim N_1^{-\frac32} \|u_1 \|_{V_S^2} M_2^{\frac{d-3}{2}} \| u_2 \|_{L_t^4 L_{x'}^{\frac{2(d-1)}{d-2}} L_{x_1}^2}
 \lambda_3^{\frac12} \| u_3 \|_{L_t^4 L_{x'}^{\frac{2(d-1)}{d-2}} L_{x_1}^2}\\
&  \lesssim \lambda_3^{\frac14} M_2^{s_c + \frac12} N_1^{- \frac74} \prod_{i=1,2,3}\|u_i \|_{V_S^2}.
\end{align*}
Next we consider the case $\lambda_1 \sim \lambda_2 \sim \lambda_3 \sim N_1$.
Since $M_2 \ll M_1 \sim \lambda_1$, we may assume that there exists $\ell \in \{2, \ldots, d\}$ such that $|\partial_{\ell} \varphi(\xi)-\partial_{\ell} \varphi(\eta)| \gtrsim N_1^2$ for $\xi \in \supp_{\xi} \widehat{u}_1$, $\eta \in \supp_{\xi} \widehat{u}_2$. 
Then, from \eqref{eq:Radtransversest1} we get
\begin{equation}
\|u_1 u_2 \|_{L_t^2 L_x^2} \lesssim M_2^{\frac{d-2}{2}} N_1^{- \frac12} \|u_1\|_{U^2_S} \|u_2\|_{U_S^2}.
\end{equation}
On the other hand, for $(\frac{1}{\alpha}, \frac{1}{\beta}) = (\frac{1}{q_1}+\frac{1}{q_2}, \frac{1}{r_1} + \frac{1}{r_2})$, we have
\begin{align}
\|u_1 u_2 \|_{L_t^{\alpha} L_{x'}^{\beta}L_{x_1}^2} & \lesssim \lambda_1^{\frac12} \|u_1\|_{L_t^{q_1} L_{x'}^{r_1} L_{x_1}^2} \|u_2\|_{L_t^{q_2} L_{x'}^{r_2} L_{x_1}^2}\\
& \lesssim \lambda_1^{\frac12 -\frac{1}{q_1}-\frac{1}{q_2}} N_1^{-\frac{d-2}{2d-3}+2 \varepsilon} \|u_1\|_{U_S^{q_1}} \|u_2\|_{U^{q_2}_S}\\
& \sim  N_1^{-\frac{1}{q_2}-\frac{d-2}{2d-3}+3 \varepsilon} \|u_1\|_{U_S^{q_1}} \|u_2\|_{U^{q_2}_S}.
\end{align}
We notice that $
\alpha, \beta\geq 1$ and $q_1,q_2>2$ if $
\varepsilon>0$ is chosen sufficiently small.
Let $\theta = \frac{2(d-1)+ 4 (2 d-3)\varepsilon}{(d-1)(2 d -5) - 4 (2 d-3)\varepsilon}$. Then, since
\[
\frac{\theta}{\alpha} + \frac{1- \theta}{2} =  \frac{1}{q_1'}, \quad
\frac{\theta}{\beta} + \frac{1- \theta}{2} = \frac{1}{r_1'},
\]
by interpolating the above two estimates (with a similar argument as in the proof of Proposition \ref{prop:d6}), we have
\begin{equation}
\|u_1 u_2 \|_{L_t^{q_1'} L_{x'}^{r_1'}L_{x_1}^2} \lesssim  M_2^{\frac{d-2}{2}(1-\theta)}
N_1^{-\frac12 - \frac{\theta}{q_2} + \frac{\theta}{2 (2 d-3)}+ 3 \varepsilon \theta}\|u_1\|_{V^2_S} \|u_2\|_{V_S^2}.
\end{equation}
This and \eqref{est:RadStrichartz-1} yield
\begin{align*}
\Bigl| \iint u_1 u_{2}
u_{3} dx dt \Bigr|
 &  \lesssim \|u_1 u_2 \|_{L_t^{q_1'} L_{x'}^{r_1'}L_{x_1}^2} \|u_{3} \|_{L_t^{q_1} L_{x'}^{r_1} L_{x_1}^2} \\
& \lesssim M_2^{\frac{d-2}{2}(1-\theta)}
N_1^{-\frac{3 d-5}{2 d-3}- \frac{\theta}{q_2} + \frac{\theta}{2 (2 d-3)}+ 3 \varepsilon (1+\theta)}
\prod_{i=1,2,3}\|u_i\|_{V^2_S}\\
& \lesssim M_2^{s_c+\varepsilon}N_1^{-1 - \varepsilon} \prod_{i=1,2,3} \|u_i \|_{V^2_S}.
\end{align*}

(3) We deal with the last case $M_1 \ll N_1$, $M_2 \ll N_1$. By symmetry, we assume $M_2 \leq M_1$. Assume first that $M_1 \gtrsim (\lambda_3 N_1)^{\frac12}$ which implies $\lambda_3 \ll \lambda_1 \sim \lambda_2 \sim N_1$. Thus, we observe that $|\partial_1 \varphi(\xi)-\partial_1 \varphi(\eta)| \gtrsim N_1^2$ for $\xi \in \supp_{\xi} \widehat{u}_2$, $\eta \in \supp_{\xi} \widehat{u}_3$. \eqref{eq:Radtransversest2} implies
\[
\|u_2 u_3 \|_{L_t^2 L_x^2} \lesssim M_{\min}^{\frac{d-1}{2}} N_1^{-1} \|u_2\|_{U^2_S} \|u_3\|_{U_S^2}.
\]
While, similarly to the above observation, we get
\[
\| u_2 u_3 \|_{L_t^{\alpha} L_{x'}^{\beta}L_{x_1}^2} \lesssim
\lambda_2^{-\frac{1}{q_1}} \lambda_3^{\frac12-\frac{1}{q_2}} M_{\min}^{-\frac{d-2}{2d-3}+2 \varepsilon} \|u_2\|_{U_S^{q_1}} \|u_3\|_{U^{q_2}_S}.
\]
Interpolating the above two, we get
\[
\|u_2 u_3 \|_{L_t^{q_1'} L_{x'}^{r_1'}L_{x_1}^2} \lesssim  \lambda_2^{-\frac{\theta}{q_1}} \lambda_3^{\frac{\theta}{2}-\frac{\theta}{q_2}} M_{\min}^{\frac{d-1}{2}(1-\theta)-\frac{d-2}{2 d-3} \theta + 2 \varepsilon \theta}  N_1^{-1+\theta} \|u_1\|_{V^2_S} \|u_2\|_{V_S^2}.
\]
Consequently, it follows from $M_1 \gtrsim \max\{M_{\min},(\lambda_3 N_1)^{\frac12}\}$ that
\begin{align*}
& \Bigl| \iint u_1 u_{2}
u_{3} dx dt \Bigr|\\
 &  \lesssim \|u_1 \|_{L_t^{q_1} L_{x'}^{r_1} L_{x_1}^2} \|u_2 u_3 \|_{L_t^{q_1'} L_{x'}^{r_1'}L_{x_1}^2}\\
& \lesssim \lambda_1^{-\frac{1+\theta}{q_1}} \lambda_3^{\frac{\theta}{2}-\frac{\theta}{q_2}}
M_1^{-\frac{d-2}{2d-3}+2 \varepsilon} M_{\min}^{\frac{d-1}{2}(1-\theta)-\frac{d-2}{2 d-3} \theta + 2 \varepsilon \theta}  N_1^{-1+\theta} \prod_{i=1,2,3}\|u_i\|_{V^2_S}\\
& \lesssim \lambda_1^{-1}\lambda_3^{\varepsilon} M_{\min}^{s_c+\varepsilon}N_1^{-2 \varepsilon} \prod_{i=1,2,3} \|u_i \|_{V^2_S}.
\end{align*}
In the case $M_1 \ll (\lambda_3 N_1)^{\frac12}$, we easily observe that at least one of $u_1$, $u_2$, $u_3$ satisfies $\supp \widehat{u}_i \subset \{ (\tau,\xi) \, |\, |\tau-\varphi(\xi)| \gtrsim \lambda_3 N_1^2\}$. 
Indeed, $M_2 \lesssim M_1 \ll (\lambda_3 N_1)^{\frac12}$ yields
\begin{align*}
& |\tau_1- \varphi(\xi)|+|\tau_2 - \varphi(\eta)|+ |\tau_1+\tau_2 - \varphi(\xi+\eta)|\\
 \gtrsim{} &
|\varphi(\xi+\eta) - \varphi(\xi) - \varphi(\eta)| \\
\geq{} & 3| \xi_1 \eta_1(\xi_1+\eta_1)| - 10(|\xi_1|+|\eta_1|)(|\xi'|^2 + |\eta'|^2) \gtrsim  \lambda_3 N_1^2,
\end{align*}
for all $\xi \in \supp_{\xi} \widehat{u}_1$, $\eta \in \supp_{\xi} \widehat{u}_2$ which satisfy $\xi+\eta \in \supp_{\xi} \widehat{u}_3$. 
In the case
$\supp \widehat{u}_1 \subset \{ (\tau,\xi)\, |\, |\tau-\varphi(\xi)| \gtrsim \lambda_3 N_1^2\}$ and $M_2 \lesssim M_3$,
since $M_2 \lesssim M_1 \ll (\lambda_3 N_1)^{\frac12}$, it follows from the Strichartz estimates \eqref{eq:linearest1} and Bernstein's inequality that
\begin{align*}
\Bigl| \iint u_1 u_{2}  u_{3} dx dt \Bigr| & \lesssim \|u_1 \|_{L_t^2 L_x^2} \|u_2 \|_{L_t^3 L_{x'}^{3(d-1)} L_{x_1}^2}  \| u_3\|_{L_t^{6} L_{x'}^{\frac{6(d-1)}{3 d-5}} L_{x_1}^{\infty}}\\
& \lesssim \lambda_3^{-\frac12} N_1^{-1} \|u_1 \|_{V_S^2} M_{2}^{\frac{d-3}{2}} \|u_2 \|_{L_t^3 L_{x'}^{\frac{6 (d-1)}{3 d -7}} L_{x_1}^2} \lambda_3^{\frac12}  \| u_3\|_{L_t^{6} L_{x'}^{\frac{6(d-1)}{3 d-5}} L_{x_1}^{2}}\\
& \lesssim \lambda_3^{-\frac16} M_1^{\frac13 + 2 \varepsilon}M_2^{s_c + \frac16-2 \varepsilon} N_1^{-\frac43} \prod_{i=1,2,3} \|u_i \|_{V_S^2}\\
& \lesssim \lambda_3^{\varepsilon} M_2^{s_c + \frac16-2 \varepsilon} N_1^{-\frac76+ \varepsilon} \prod_{i=1,2,3} \|u_i \|_{V_S^2}.
\end{align*}
The other cases are treated similarly.
\end{proof}
As above, Theorem \ref{thm:d4rad} follows by the standard argument.
\bibliographystyle{amsplain}
\bibliography{zkhigh}
\end{document}